\documentclass[11pt]{article}
\usepackage{amsmath, amsfonts, amssymb, mathtools}
\usepackage{graphicx, amsthm, color}
\usepackage{pgfplots, caption, subcaption, tikz, tikz-cd} 
\usepackage{float}
\usepackage{enumitem}
\usepackage{cleveref}
\usepackage{euscript}
\usepackage{mathrsfs}

\usepackage[scr=boondox]{mathalpha}

\usepackage[T1]{fontenc}
\usepackage{fbb}

\newtheorem{theorem}{Theorem}

\newtheorem{lemma}[theorem]{Lemma}
\newtheorem{proposition}[theorem]{Proposition}

\newtheorem{claim}[theorem]{Claim}

\theoremstyle{definition}

\newtheorem{remark}[theorem]{Remark}
\newtheorem{fact}[theorem]{Fact}

\newcommand{\R}{\mathbb{R}}
\newcommand{\N}{\mathbb{N}}

\newcommand{\vol}{\mathrm{vol}}

\newcommand{\rad}{\mathscr{r}}

\begin{document}

\begin{center}
{\LARGE Impossibility of a nontrivial Brunn--Minkowski inequality for \\ higher Dirichlet eigenvalues}
\end{center}
\smallskip
\begin{center}
{\Large \textsc{Tr\'i Minh L\^e \& Khai-Hoan Nguyen-Dang}}
\end{center}

\bigskip

\noindent\textbf{Abstract.}
Let $\lambda_j(K)$ be the $j$th Dirichlet eigenvalue of a convex body $K$.
It is well known that $\lambda_1$ satisfies a Brunn--Minkowski inequality: $K \mapsto \lambda_1(K)^{-1/2}$ is concave on the family of convex bodies.
We show that no analogous statement holds for higher eigenvalues.
More precisely, for any $j \geq 2$ and $N \geq 2$, if $K \mapsto (f \circ \lambda_j)(K)$ is concave on the family of convex bodies in $\R^N$ for some function $f: (0,  \infty) \to \R$, then $f$ must be constant.

\bigskip


\vspace{0.6cm}


\section{Introduction}
Let $\vol(K)$ be the volume of a measurable set $K \subset \R^N$.
A set $K$ is a convex body if it is a compact convex set with nonempty interior. 
The classical Brunn--Minkowski inequality states that $K \mapsto \vol(K)^{1/N}$ is concave on the family of convex bodies:
\begin{equation}
    \vol(tK + (1 - t)L)^{1/N} \geq t \vol(K)^{1/N} + (1 - t) \vol(L)^{1/N},
\end{equation}
for every convex bodies $K, L \subset \R^N$ and $t \in [0, 1]$.
This inequality is a cornerstone of convex geometry and related geometric optimization problems, we refer the reader to \cite{G_2002} for a comprehensive survey.

\medskip 

Analogues of the Brunn--Minkowski inequality have been found for other important variational quantities, notably the first Dirichlet eigenvalue, capacity and the torsion/Saint--Venant functional, see \cite{C_2005} and references therein.
In particular, the pioneering work of Brascamp and Lieb \cite{BL_1976} established the concavity of the map $K \mapsto \lambda_1(K)^{-1/2}$:
\begin{equation}\label{BM_lambda1}
    \lambda_1(t K + (1 - t)L)^{-1/2} \geq t \lambda_1(K)^{-1/2} + (1 - t) \lambda_1(L)^{-1/2},
\end{equation}
for every convex bodies $K, L \subset \R^N$ and $t \in [0, 1]$.
Here, $\lambda_1$ denotes the Dirichlet first eigenvalue on $K$.
This phenomenon has subsequently been observed for the first eigenvalue of other operators, for instance, general homogeneous elliptic operators \cite{CF_2020}, elliptic operators in Gauss spaces \cite{ CFLS_2024, CQS_2026} and the Monge--Amp\`ere operator \cite{S_2005}.

\medskip

Since all Dirichlet eigenvalues have the same homogeneity under dilations, the quantity $\lambda_j(K)^{-1/2}$ has the same scaling as a length. 
Thus \eqref{BM_lambda1} suggests, at least formally, the possibility of analogous Brunn--Minkowski inequalities for higher eigenvalues.
A natural question then arises:
\begin{center}
    \textit{Can one obtain a Brunn--Minkowski type inequality for higher Dirichlet eigenvalues $\lambda_j$ with $j \geq 2$?}
\end{center}
In contrast with the mentioned works, Bucur, Fragal\`a and Lamboley~\cite[Proposition 2.6]{BFL_2012} have shown that
\begin{itemize}
    \item in dimension $2$, the inequality~\eqref{BM_lambda1} fails when $\lambda_1$ is replaced by the second eigenvalue $\lambda_2$.
\end{itemize}

\medskip

This leaves open, however, whether the obstruction is specific to the canonical homogeneous choice $s \mapsto s^{-1/2}$, or whether higher eigenvalues are incompatible with Brunn--Minkowski concavity in a more intrinsic sense.
Our main result shows that the latter is the case.
More precisely, Theorem~\ref{thm.noBM} proves that no nontrivial analogue of the Brunn--Minkowski inequality can hold for higher Dirichlet eigenvalues, even after an arbitrary scalar reparametrization:
\begin{itemize}
    \item for any $j \geq 2$ and $N \geq 2$, if $K \mapsto (f \circ \lambda_j)(K)$ is concave on the family of convex bodies in $\R^N$ for some function $f: (0,  \infty) \to \R$, then $f$ must be constant.
\end{itemize}
The main ingredient is a rectangular construction in which two boxes have the same $j$th eigenvalue, while the $j$th eigenvalue of their midpoint can be prescribed to be either larger or smaller.
As a consequence of our construction, we also obtain a characterization of Brunn--Minkowski inequalities associated with $\lambda_1$, see Proposition~\ref{prop.BM1}.

\paragraph{Notations and basic facts.}
Throughout,  $\lambda_j(K) := \lambda_j(\mathrm{int} K)$ denotes the $j$th Dirichlet eigenvalue on a convex body $K$.
Denote $\N := \{ 1, 2, \cdots \}$.
For $\pmb{a} := (a_1, \cdots, a_N) \in (0, \infty)^N$ and $\pmb{m} = (m_1, \cdots, m_N) \in \N^N$, denote $R_{\pmb{a}} := [0, a_1] \times \cdots \times [0, a_N]$ and
\[
    \lambda_{\pmb{m}}(R_{\pmb{a}}) := \pi^2 \sum_{\ell = 1}^N \dfrac{m_\ell^2}{a_\ell^2}.
\]
Let us recall the following facts concerning the Dirichlet eigenvalues on rectangular boxes.
\begin{fact}\label{fa.EiOnBox}(see \cite[Example 6.1]{B_2020})
Let $N \geq 1$ and fix a rectangular box $R_{\pmb{a}}$ with $\pmb{a} = (a_1, \cdots, a_N) \in (0, \infty)^N$.
The Dirichlet eigenvalues on $R_{\pmb{a}}$ are obtained through the nondecreasing rearrangement, counting with multiplicity, of the family
\begin{align*}
	\big\{  \lambda_{\pmb{m}}(R_{\pmb{a}})  : \, \pmb{m} = (m_1, \cdots, m_N) \in \mathbb{N}^N \big\}.
\end{align*}
\end{fact}

\begin{fact}\label{fa.EiHo}(see \cite[Section 1.2.3]{H_2006})
For any convex body $K$, $\lambda_j(tK) = t^{-2} \lambda_j(K)$ for every $j \geq 1$ and $t > 0$.
\end{fact}

\section{Main results}



This section proves the following result: for higher Dirichlet eigenvalues, Brunn--Minkowski concavity cannot be restored on the family of convex bodies by any nonconstant scalar reparametrization.

\begin{theorem}\label{thm.noBM}
Let $j \geq 2$ and $N \geq 2$.
Let $f : (0, \infty) \to \R$ be such that 
\begin{equation}\label{f_concave}
    f(\lambda_j(t K + (1 - t)L)) \geq t f(\lambda_j(K)) + (1 - t) f(\lambda_j(L)),
\end{equation}
for every convex bodies $K, L \subset \R^{N}$ and $t \in [0, 1]$.
Then, $f$ is constant.
\end{theorem}

To prove the above theorem, let us start by some preparatory lemmas.

\begin{lemma}\label{lem.continuity}
    Let $j \geq 1$.
    Then, the function $\nu_j : (x, y) \mapsto \lambda_j(R_{(x, y)})$ is continuous on $(0, \infty) \times (0, \infty)$.
\end{lemma}

\begin{proof}
Fix $x_0, y_0 > 0$ and set $D = [x_0/2, 3x_0/2] \times [y_0/2, 3y_0/2]$.
Since the eigenvalue at $(m_1,m_2)$--mode is determined by
\[
    \lambda_{(m_1, m_2)}(R_{(x, y)}) = \pi^2 \left( \dfrac{m_1^2}{x^2} + \dfrac{m_2^2}{y^2} \right),
\]
the positive lower bound for $(x, y) \in D$ then yields
\[
    \inf_{(x, y) \in D} \lambda_{(m_1, m_2)}(R_{(x, y)}) \to \infty \quad \text{ as } m_1^2 + m_2^2 \to \infty.
\]
Then, there exists $M \in \N$ large enough such that 
\begin{equation}\label{Big_Eig}
    \lambda_{(m_1, m_2)} (R_{(x, y)}) \geq \nu_j(x_0, y_0) + 1 \quad \text{ for every } (x, y) \in D, \,\, m_1^2 + m_2^2 \geq M.
\end{equation}
Set
\[ 
    F := \big\{ (m_1, m_2) \in \N^2 : m_1^2 + m_2^2 < M \big\}.
\]
It follows from~\eqref{Big_Eig} that the first $j$ eigenvalues at $(x_0, y_0)$ are indexed by elements of $F$.
Denote
\[
    \mu_j(x, y) := \min_{\substack{A \subset F \\ \# A = j} } \max_{(m_1, m_2) \in A} \lambda_{(m_1, m_2)}(R_{(x, y)})
\]
the  $j$th smallest value of the finite family $\{ \lambda_{(m_1, m_2)}(R_{(x, y)}) \}_{(m_1, m_2) \in F}$.
Observe that $\mu_j(x_0, y_0) = \nu_j(x_0, y_0)$ and $\mu_j$ is continuous since it is obtained from finitely many continuous functions by taking minima and maxima. 

\medskip

By the continuity of $\mu_j$, there exists a neighborhood $U \subset D$ of $(x_0, y_0)$ such that 
\begin{equation}\label{U_Neigh}
    \mu_j(x, y) < \nu_j(x_0, y_0) + 1 \quad \text{ for every } (x, y) \in U.
\end{equation}
Combining~\eqref{Big_Eig} and~\eqref{U_Neigh}, we obtain
\begin{align*}
    \lambda_{(m_1, m_2)}(R_{(x, y)}) > \mu_j(x, y) \quad \text{ for every } (x, y) \in U, \,\, (m_1, m_2) \not\in F.
\end{align*}
This implies that for any $(x, y) \in U$, no eigenvalue indexed by $(m_1, m_2) \not\in F$ can contribute to the first 
$j$ eigenvalues of $R_{(x, y)}$.
Hence, $\mu_j = \nu_j$ on $U$.
Then, the continuity of $\mu_j$ implies the continuity of $\nu_j$ at $(x_0, y_0)$, completing the proof.
\end{proof}

\begin{lemma}\label{lem.2D_wierdBox+}
    Let $j \geq 2$.
    Then, there exists $\rad_+  > 1$ such that for every $\sigma \in [1, \rad_+]$, there are planar rectangles $K, L$ satisfying 
        \[
            \lambda_j(K) = \lambda_j(L) \quad \text{ and } \quad \lambda_j \left( \dfrac{K + L}{2} \right) = \sigma \lambda_j(K).
        \]
\end{lemma}

\textbf{Idea of the proof.}  
The idea is to start from a rectangle for which the two modes $(j,1)$ and $(1,2)$ coincide and realize the $j$-th eigenvalue, while the modes $(1,1),\ldots,(j-1,1)$ lie strictly below.
We then perturb this rectangle in
two different directions while preserving the value of the $j$-th eigenvalue for each perturbed rectangle.
The midpoint of the two resulting rectangles has a smaller second side length, and consequently its $j$-th eigenvalue is strictly larger.

\begin{proof}[Proof of Lemma~\ref{lem.2D_wierdBox+}]
Denote $\nu_j(x, y) := \lambda_j(R_{(x, y)})$ for $x, y > 0$.
We suppress the common factor $\pi^2$ and write the normalized Dirichlet eigenvalues of the rectangle with side lengths $x, y > 0$ as
\begin{align*}
    \theta_{m_1, m_2}(x, y) := \dfrac{m_1^2}{x^2} + \dfrac{m_2^2}{y^2}, \quad \text{ for every } m_1, m_2 \in \N.
\end{align*}
Set 
\[
    A := \sqrt{ \dfrac{4j^2 - 1}{3} } \quad \text{ and } \quad B := \sqrt{ \dfrac{4j^2 - 1}{j^2 - 1} }.
\]
Note that $A > j > 1$ since $j \geq 2$.
A direct computation yields
\begin{equation}\label{BaseID}
    \dfrac{ j^2 }{ A^2 } + \dfrac{1}{B^2} = 1 \,\, \text{ and  } \,\, \dfrac{1}{A^2} + \dfrac{2^2}{B^2} = 1,
\end{equation}
and so $\theta_{j, 1}(A, B) = \theta_{1, 2}(A, B) = 1$.
On the rectangle $R_{(A, B)}$, we consider the following cases:

\medskip

\textbf{Case 1:}
\emph{$m_1 < j$ and $m_2 = 1$}.
In this case, one has 
\[
    \theta_{m_1, 1}(A, B) < \theta_{j, 1}(A, B) = 1.
\]

\smallskip

\textbf{Case 2:}
\emph{$m_1 > j$ and $m_2 = 1$.}
In this case, one has
\[
    \theta_{m_1, 1}(A, B) > \theta_{j, 1}(A, B) = 1.
\]

\smallskip

\textbf{Case 3:}
\emph{$m_2\geq 2$ and $(m_1,m_2)\neq (1,2)$}.
In this case, one has
\[
\theta_{m_1,m_2}(A,B)> \theta_{1, 2}(A, B) = 1.
\]

\smallskip

As a consequence of the above observations, on $R_{(A, B)}$, the first $j - 1$ eigenvalues are below $\pi^2$ and the $j$th eigenvalue is exactly $\pi^2$, that is, $\nu_j(A, B) = \pi^2$.

\medskip

Next, we will perturb the rectangle $R_{(A, B)}$ in two opposite directions such that the $j$th eigenvalue of the resulting perturbed rectangles remains equal to $\pi^2$.
Set
\begin{align*}
    & \, \varphi_v(x) := \dfrac{2}{\sqrt{1 - x^{- 2} }}, \quad\quad \text{ for } x > 1, \\
    & \, \varphi_h(x) := \dfrac{1}{\sqrt{1 - j^2 x^{-2} }}, \quad \text{ for } x > j.
\end{align*}
Notice that by~\eqref{BaseID}, $\varphi_v(A) = \varphi_h(A) = B$.

\medskip

We first claim that $\nu_j(A - \delta, \varphi_v(A - \delta)) = \pi^2$ for every $\delta > 0$ small enough.
Indeed, by definition of $\varphi_v$, we have
\begin{align}
    & \, \theta_{1, 2} (x, \varphi_v(x)) = 1, \qquad\qquad\qquad\, \text{ for } x > 1, \label{Vertical_Node[12]}\\
    &\, \theta_{j, 1}(x, \varphi_v(x)) = \dfrac{1}{4} + \dfrac{j^2 - 1/4}{x^2}, \quad \text{ for } x > 1 \label{Vertical_Node[j1]}.
\end{align}
Notice that
\[
    \theta_{j - 1, 1}(A, \varphi_v(A)) = \theta_{j - 1, 1}(A, B) = 1 - \dfrac{3}{2j + 1} < 1,
\]
and 
\[
    \theta_{1, 1}(x, y) \leq \cdots \leq \theta_{j - 1, 1}(x, y) \quad \text{ for } x, y > 0.
\]
Consequently, thanks to the continuity of the map $x \mapsto \theta_{j - 1, 1}(x, \varphi_v(x))$, there exists $\delta_0 \in (0, A - 1)$ small enough such that
\begin{equation}\label{Ver.Nei.01}
    \theta_{m_1, 1}(A - \delta, \varphi_v(A - \delta)) < 1 \,\, \text{ for every } 0 < \delta < \delta_0, \, m_1 \in \{1, \cdots, j - 1 \}.
\end{equation}
Consider the following cases:

\smallskip
\textbf{Case 1:}
\emph{$m_1 \geq j$ and $m_2 = 1$.}
Due to~\eqref{Vertical_Node[j1]}, the map $x \mapsto \theta_{j, 1}(x, \varphi_v(x))$ is decreasing on $(1, \infty)$ and so 
\[
    \theta_{j, 1}(A - \delta, \varphi_v(A - \delta)) > \theta_{j, 1}(A, \varphi_v(A)) = \theta_{j, 1}(A, B) = 1.
\]
Therefore, for any $m_1 \geq j$, it holds
\begin{equation}\label{Ver.Nei.02}
    \theta_{m_1, 1}(A - \delta, \varphi_v(A - \delta)) \geq \theta_{j, 1}(A - \delta, \varphi_v(A - \delta)) > 1 \,\, \text{ for every } 0 < \delta < \delta_0.
\end{equation}

\smallskip

\textbf{Case 2:}
\emph{$(m_1, m_2) \neq (1, 2)$  and $m_2 \geq 2$.}
In this case, it follows from~\eqref{Vertical_Node[12]} that 
\begin{equation}\label{Ver.Nei.03}
    \theta_{m_1, m_2}(A - \delta, \varphi_v(A - \delta)) > \theta_{1, 2}(A - \delta, \varphi_v(A - \delta)) = 1.
\end{equation}

\medskip

Combining~\eqref{Vertical_Node[12]}, \eqref{Ver.Nei.01}, \eqref{Ver.Nei.02} and \eqref{Ver.Nei.03}, on $R_{(A - \delta, \varphi_v(A - \delta))}$ for every $0 < \delta < \delta_0$, the first $j - 1$ eigenvalues are below $\pi^2$  (counting with multiplicity) and the $j$th eigenvalue is exactly $\pi^2$, corresponding to the $(1, 2)$-mode.
Therefore, we have proved that
\begin{equation}\label{PertuVer}
    \nu_j(A - \delta, \varphi_v(A - \delta)) = \pi^2 \quad \text{ for every }0 < \delta < \delta_0.
\end{equation}

\medskip

Analogously, there exists $\delta_1 > 0$ such that
\begin{equation}\label{PertuHo}
    \nu_j(A + \delta, \varphi_h(A + \delta)) = \pi^2 \quad \text{ for every } 0 < \delta < \delta_1.
\end{equation}
For any fixed $\delta \in [0, \min \{\delta_0, \delta_1\})$, denote
\[
    K_{\delta} := R_{(A - \delta, \varphi_v(A - \delta))}, \,\, L_\delta := R_{(A + \delta, \varphi_h(A + \delta))} \,\, \text{ and } \,\, M_\delta := \dfrac{1}{2}(K_\delta + L_\delta) = R_{(A, \bar\varphi(\delta))},
\]
where
\[
    \bar \varphi(\delta) := \dfrac{1}{2} \big( \varphi_v(A - \delta) + \varphi_h(A + \delta) \big).
\]

\medskip

We now show that $\nu_j(A, \bar\varphi(\delta)) > \pi^2$ for every $\delta > 0$ small enough.
Indeed, a direct computation gives
\begin{align*}
    & \varphi_v'(x) = - \dfrac{2}{ x^3(1 - x^{-2})^{3/2} } \quad\, \text{ for } x > 1, \\
    & \varphi_h'(x) = - \dfrac{j^2}{x^3(1 - j^2 x^{-2})^{3/2} } \,\, \text{ for } x > j.
\end{align*}
Since $\varphi_v(A) = \varphi_h(A) = B$, we get
\[
    \varphi_h'(A)  - \varphi'_v(A) = - \left( j^2 - \dfrac{1}{4} \right) \dfrac{B^3}{A^3} < 0.
\]
Taylor expansion at $\delta = 0$ yields
\[
    \bar \varphi(\delta) = B + \dfrac{\delta}{2} (\varphi_h'(A)  - \varphi'_v(A)) + o(\delta).
\]
Hence, there exists $\delta_\ast \in (0, \min \{ \delta_0, \delta_1\})$ such that $\bar \varphi(\delta) < B$ for  every $\delta \in (0, \delta_\ast]$.
For such $\delta > 0$, we consider the following cases:

\smallskip

\textbf{Case 1:} 
\emph{$m_1 \geq j$ and $m_2 = 1$.}
In this case, we get
\begin{align*}
    \theta_{m_1, 1}(A, \bar\varphi(\delta)) \geq \theta_{j, 1}(A, \bar \varphi(\delta))  = \dfrac{j^2}{A^2} + \dfrac{1}{\bar\varphi(\delta)^2} > \dfrac{j^2}{A^2} + \dfrac{1}{B^2} 
    = 1,
\end{align*}
where we have used ~\eqref{BaseID} in the last identity.
\smallskip

\textbf{Case 2:}
\emph{$m_1 \geq 1$ and $m_2 \geq 2$.}
In this case, we get
\begin{align*}
    \theta_{m_1, m_2}(A, \bar\varphi(\delta)) \geq \theta_{1, 2}(A, \bar\varphi(\delta)) = \dfrac{1^2}{A^2} + \dfrac{2^2}{\bar\varphi(\delta)^2} > \dfrac{1^2}{A^2} + \dfrac{2^2}{B^2} 
    = 1,
\end{align*}
where we have used ~\eqref{BaseID} in the last identity.
\smallskip

Thus, on the rectangle $R_{(A, \bar\varphi(\delta))}$, at most the $j - 1$ modes
\[
    (1, 1), (2, 1), \cdots, (j - 1, 1)
\]
can have normalized value at most $1$.
Therefore, the $j$th normalized eigenvalue is strictly larger than $1$ and so
\begin{equation}\label{MLarge}
    \nu_j(A, \bar\varphi(\delta)) > \pi^2 \quad \text{ for every $\delta\in (0, \delta_\ast]$.}
\end{equation}

\textbf{Conclusion.}
At $\delta = 0$, we know that $K_0 = L_0 = R_{(A, B)}$ and so $\lambda_j(K_0) = \lambda_j(L_0) = \pi^2$.
Combining with~\eqref{PertuVer} and~\eqref{PertuHo}, we get that $\lambda_j(K_\delta) = \lambda_j(L_\delta) = \pi^2$ for every $\delta \in [0, \delta_\ast]$.
Define $\rho_+ : [0, \delta_\ast] \to (0, + \infty)$ by
\[
    \rho_+(\delta) = \dfrac{\lambda_j(M_\delta)}{\lambda_j(K_\delta)}.
\]
Thanks to Lemma~\ref{lem.continuity}, $\rho_+$ is continuous on $[0, \delta_\ast]$. 
Moreover, observe that $\rho_+(0) = 1$ and by~\eqref{MLarge},
$\rho_+(\delta) > 1$ for every $\delta \in (0, \delta_\ast]$.
Set 
\[
    \rad_+ := \max_{\delta \in [0, \delta_\ast]}\rho_+(\delta) > 1.
\]
Applying the intermediate value theorem, for any $\sigma \in [1, \rad_+]$, there exists $\delta \in [0, \delta_\ast]$ such that $\rho_+(\delta) = \sigma$.
Consequently, for this choice of $\delta$,
\[
    \lambda_j(M_\delta) = \rho_+(\delta) \lambda_j(K_\delta) = \sigma \lambda_j(K_\delta).
\]
This completes the proof.
\end{proof}

\begin{lemma}\label{lem.2D_wierdBox-}
    Let $j \geq 2$.
    Then, there exists $\rad_- \in (0, 1)$ such that for every $\sigma \in [\rad_-, 1]$, there are planar rectangles $K, L$ satisfying 
        \[
            \lambda_j(K) = \lambda_j(L) \quad \text{ and } \quad \lambda_j \left( \dfrac{K + L}{2} \right) = \sigma \lambda_j(K).
        \]
\end{lemma}

\begin{proof}
Denote $\nu_j(x,y):=\lambda_j(R_{(x,y)})$ for $x,y>0$.
For every $a\geq 1$, since $R_{(a,1)}$ and $R_{(1,a)}$ are congruent, we have $   \nu_j(a,1)=\nu_j(1,a)$.
Moreover,
\[
    \dfrac{1}{2} (R_{(a,1)}+R_{(1,a)}) =     R_{\left(\frac{a+1}{2},\frac{a+1}{2}\right)}.
\]
The homogeneity of $\lambda_j$ then gives
\[
    \nu_j\left(\dfrac{a+1}{2},\dfrac{a+1}{2}\right)
    =
    \dfrac{4\nu_j(1,1)}{(a+1)^2}.
\]
Define $\rho_-:[1,+\infty)\to(0,+\infty)$ by
\[
    \rho_-(a)
    :=
    \dfrac{\nu_j\left((a + 1)/2, (a  + 1)/2 \right)}{\nu_j(a,1)}
    =
    \dfrac{4\nu_j(1,1)}{(a+1)^2\nu_j(a,1)}.
\]
Notice that $\rho_-(1)=1$. 
We claim that $\rho_-(a)\to 0$ as $a\to+\infty$.
Indeed, using Fact~\ref{fa.EiOnBox}, 
\[
    \nu_j(a,1)\geq \pi^2
    \quad \text{ for every } a\geq 1.
\]
Consequently,
\[
    0<\rho_-(a)
    \leq
    \dfrac{4\nu_j(1,1)}{\pi^2(a+1)^2}
    \to 0
    \quad \text{as } a\to+\infty.
\]
Choose $a_\ast>1$ large enough such that $\rho_-(a_\ast)<1$ and set $\rad_-:=\rho_-(a_\ast)\in(0,1)$.
Thanks to Lemma~\ref{lem.continuity}, the map $\rho_-$ is continuous on $[1,a_\ast]$.
Therefore, applying the intermediate value theorem, for every $\sigma\in[\rad_-,1]$, there exists $a\in[1,a_\ast]$ such that $\rho_-(a)=\sigma$.
For this choice of $a$, set
\[
    K:=R_{(a,1)}
    \quad \text{and} \quad
    L:=R_{(1,a)}.
\]
Then, we get
\[
    \lambda_j(K)=\lambda_j(L) \quad \text{ and } \quad \lambda_j\left(\dfrac{K+L}{2}\right)
    =
    \rho_-(a)\lambda_j(K)
    =
    \sigma\lambda_j(K).
\]
This completes the proof.
\end{proof}

From the two lemmas above, we obtain the following key property.

\begin{proposition}\label{prop.wierdBox}
Let $N\ge2$ and $j\ge2$.  
There exists $\rad_0 > 1$ such that for every
$\sigma\in[\rad_0^{-1},\rad_0]$ there are rectangular convex bodies $K ,L 
\subset\R^N$ satisfying
\begin{equation}\label{eq:equal-endpoints}
 \lambda_j(K )=\lambda_j(L ) \quad \text{ and } \quad \lambda_j\!\left(\frac{K + L }{2}\right)=\sigma\lambda_j(K).
\end{equation}
\end{proposition}

\begin{proof}
Let $\rad_+ > 1$ and $\rad_- \in (0, 1)$ be defined as in Lemma~\ref{lem.2D_wierdBox+} and Lemma~\ref{lem.2D_wierdBox-}, respectively.
If $N = 2$, taking $\rad_0 := \min \{ \rad_+, \rad_-^{-1} \} > 1$ yields the desired conclusion.

\medskip

Assume $N > 2$.
Fix $\rad > 0$ sufficiently large such that 
\[
    \rad > \dfrac{(N - 2) \rad_+}{3}.
\]
Set
\[
    \rad_0 := \min \left\{ 1 + \dfrac{\rad_+ - 1}{1 + \rad}, \dfrac{1 + \rad}{\rad + \rad_-}  \right\} > 1.
\]
Given $\sigma \in [\rad_0^{-1}, \rad_0]$, set $\tau := \sigma(1 + \rad) - \rad$.
It is straightforward to check that
\begin{itemize}
    \item if $\sigma \in [1, \rad_0]$, then $\tau \in [1, \rad_+]$;
    \item if $\sigma \in [\rad_0^{-1}, 1]$, then $\tau \in [\rad_-, 1]$.
\end{itemize}
Therefore, it follows from Lemmas~\ref{lem.2D_wierdBox+}--\ref{lem.2D_wierdBox-} that there exist planar rectangles $K_\tau, L_\tau$ such that 
\[
    \lambda_j(K_\tau) = \lambda_j(L_\tau) =: \alpha \quad \text{ and } \quad \lambda_j \left(  \dfrac{K_\tau + L_\tau }{2}\right) = \tau \alpha.
\]
Take $\epsilon > 0$ such that $\lambda_1(Q) = \rad \alpha$ where $Q := [0, \epsilon]^{N - 2}$.
Set
\begin{equation}\label{KL_Choice}
    K := K_\tau \times Q \quad \text{ and } \quad L := L_\tau \times Q.
\end{equation}
To compute the $j$th eigenvalue of $K$ and $L$, we need the following elementary claim, which is a consequence of Fact~\ref{fa.EiOnBox}.

\begin{claim}\label{claim.EigOnProduct}
    Let $A \subset \R^2$ be a rectangle and let $P \subset \R^{N - 2}$ be a rectangular box.
    If $\lambda_2(P) - \lambda_1(P) > \lambda_j(A)$,
    then
    \[
            \lambda_i(A \times P) = \lambda_i(A)  + \lambda_1(P) \quad \text{ for every } i \in \{ 1, \cdots, j \}.
    \]
\end{claim}

\textit{Proof of Claim~\ref{claim.EigOnProduct}.}
Thanks to Fact~\ref{fa.EiOnBox}, the Dirichlet spectrum on $A \times P$ is the nondecreasing rearrangement, counting with multiplicity, of the set
\[
    \big\{ \lambda_p(A) + \lambda_q(P) : p, q \geq 1 \big\}.
\]
Among the modes with $q  = 1$, the first $j$ eigenvalues are 
\begin{equation}\label{AQ_jfirst}
    \lambda_1(A) + \lambda_1(P) \leq \cdots \leq \lambda_j(A) + \lambda_1(P).
\end{equation}
Considering the case $q \geq 2$, since $\lambda_2(P) - \lambda_1(P) > \lambda_j(A) > \lambda_j(A) - \lambda_1(A)$, we obtain
\[
    \lambda_p(A) + \lambda_q(P) \geq \lambda_1(A) + \lambda_2(P) > \lambda_j(A) + \lambda_1(P) \quad \text{ for every } p \geq 1.
\]
Therefore, no eigenvalue of the $(p, q)$-mode with $q \geq 2$ contributes to the first $j$ eigenvalues of $A \times P$.
Lastly, by~\eqref{AQ_jfirst}, we conclude
\[
    \lambda_i(A \times P) = \lambda_i(A) + \lambda_1(P) \quad \text{ for every } i \in \{1, \cdots, j \}.
\]
Claim~\ref{claim.EigOnProduct} is proven.
\hfill$\Diamond$

\medskip

Coming back to our proof, it follows from Fact~\ref{fa.EiOnBox} that
\[
    \lambda_1(Q) = \dfrac{\pi^2(N - 2)}{\epsilon^2} \quad \text{ and } \quad \lambda_2(Q) = \dfrac{\pi^2(N + 1)}{\epsilon^2}
\]
and hence, by the choice of $\rad$,
\[
    \lambda_2(Q) - \lambda_1(Q) = \dfrac{3\pi^2}{\epsilon^2} = \dfrac{3 \rad \alpha}{N - 2} > \alpha = \lambda_j(K_\tau).
\]
Applying Claim~\ref{claim.EigOnProduct} to the case $A = K_\tau$ (resp. $A = L_\tau$) and $P = Q$, we infer that
\[
    \lambda_j(K) = \lambda_j(K_\tau) + \lambda_1(Q) = (\rad + 1)\alpha \quad \text{ (resp. $\lambda_j(L) = (\rad + 1)\alpha$)}.
\]
Notice that
\[
    \dfrac{1}{2} (K + L) = \dfrac{K_\tau + L_\tau}{2} \times Q.
\]
Again, by the choice of $\rad$, we have
\[
    \lambda_2(Q) - \lambda_1(Q) = \dfrac{3\rad \alpha}{N - 2} > \rad_+ \alpha \geq \tau \alpha = \lambda_j \left( \dfrac{K_\tau + L_\tau}{2}  \right).
\]
Then, applying Claim~\ref{claim.EigOnProduct} to the case $A = (K_\tau + L_\tau)/2$ and $P = Q$ gives
\[
    \lambda_j \left( \dfrac{K + L}{2} \right) = \lambda_j \left( \dfrac{K_\tau + L_\tau}{2}  \right) + \lambda_1(Q) = (\tau + \rad)\alpha = \sigma \lambda_j(K).
\]
Therefore, the sets $K, L$ chosen by~\eqref{KL_Choice} satisfy the required properties, which completes the proof.
\end{proof}

We are now ready to prove the main results.

\begin{proof}[Proof of Theorem~\ref{thm.noBM}]
Let $\rad_0 > 1$ be defined as in Proposition~\ref{prop.wierdBox}.
Fix $\sigma \in [\rad_0^{-1}, \rad_0]$ and $x > 0$.
Thanks to Proposition~\ref{prop.wierdBox}, there exist rectangular convex bodies $K, L$ such that
\[
    \lambda_j(K) = \lambda_j(L) \quad \text{ and } \lambda_j \left( \dfrac{K + L}{2} \right) = \sigma \lambda_j(K)
\]
Set $s := \sqrt{\lambda_j(K)/x}$.
It follows that
\[
    \lambda_j(sK) = \lambda_j(sL) = x \quad \text{ and } \lambda_j \left( \dfrac{sK + sL}{2} \right) = \sigma x
\]
Applying the inequality~\eqref{f_concave} to the convex bodies $sK$, $sL$ and $t = 1/2$, we obtain $f(\sigma x) \geq f(x).$
Since $x > 0$ and $\sigma \in [\rad_0^{-1}, \rad_0]$ are arbitrary, we infer
\begin{equation}\label{f-sigma}
	f(\sigma x) \geq f(x) \quad \text{ for every } x > 0, \, \sigma \in [\rad_0^{-1}, \rad_0].
\end{equation}
Notice that if $\sigma \in [\rad_0^{-1}, \rad_0]$, then $\sigma^{-1} \in [\rad_0^{-1}, \rad_0]$.
Applying~\eqref{f-sigma} again, we get
\begin{align*}
	f(x) = f(\sigma^{-1}\sigma x) \geq f(\sigma x),
\end{align*}
and therefore we arrive at
\begin{equation}\label{f-sigm-Eqal}
	f(\sigma x) = f(x) \quad \text{ for every } x > 0, \,\, \sigma \in [\rad_0^{-1}, \rad_0].
\end{equation}

\medskip

Lastly, fix $x, y > 0$.
Choose $d \in \N$ sufficiently large such that 
\[
	\sigma = \left( \dfrac{y}{x}\right)^{1/d} \in [\rad_0^{-1}, \rad_0].
\]
Note that $\sigma^d x = y$.
Applying the identity~\eqref{f-sigm-Eqal} $d$ times,  we obtain
\[
	f(y) = f(\sigma^d x) = f(\sigma^{d - 1} x) = \cdots = f(\sigma x)  = f(x).
\]
Thus $f(x) = f(y)$ for all $x, y > 0$ and so $f$ is constant.
\end{proof}

As a consequence of the preceding techniques, we conclude the paper with a result for the first eigenvalue, which extends the classical Brascamp--Lieb inequality.

\begin{proposition}\label{prop.BM1}
    Let $N \geq 2$ and $f: (0, \infty) \to \R$.
    Then, the following assertions are equivalent:
    \begin{itemize}
        \item[(a)] the function $f \circ \lambda_1$ is concave on the family of convex bodies in $\R^N$;

        \item[(b)]
        the function $g(r) := f(r^{-2})$ is concave and nondecreasing on $(0, \infty)$.
    \end{itemize}
\end{proposition}

\begin{remark}
The restriction $N\ge 2$ is essential.
Indeed, in dimension one, convex bodies are nondegenerate intervals.
Consider $f(r) := - r^{-1/2}$ for $r > 0$.
Then $g(r)=-r$, which is concave and strictly decreasing on $(0,\infty)$.
On the other hand,
\[
    f(\lambda_1(I)) = -\lambda_1(I)^{-1/2} =  -\frac{|I|}{\pi}.
\]
If $I,J\subset\mathbb R$ are nondegenerate intervals, then
\[
    |tI+(1-t)J|
    =
    t|I|+(1-t)|J|,
    \qquad t\in[0,1].
\]
Hence $f\circ\lambda_1$ is affine and in particular concave, on the family of  nondegenerate intervals.
Thus, in dimension one, the concavity of $f\circ\lambda_1$ does not force $g$ to be nondecreasing.
\end{remark}

\begin{proof}[Proof of Proposition~\ref{prop.BM1}]
Assume that $(b)$ holds.
It is not hard to infer from the Brascamp--Lieb inequality that $f \circ \lambda_1$ is concave on the family of convex bodies in $\R^N$.

\medskip

Assume that $(a)$ holds. 
We first prove that $g$ is concave. 
Fix a convex body $C$ and without any loss of generality, assume that $\lambda_1(C) = 1$.
Fix $r,s>0$ and $t\in[0,1]$.
Applying the concavity inequality of $f \circ \lambda_1$ to the case $K = r C$ and $L = sC$ and using the homogeneity of $\lambda_1$, we obtain
\begin{align*}
    g(tr + (1 - t)s) = f(\lambda_1(tK + (1 - t)L)) \geq t f(\lambda_1(K)) + (1 - t) f(\lambda_1(L)) = tg(r) + (1 - t)g(s).
\end{align*}

\medskip

It remains to show that $g$ is nondecreasing. 
For any fixed $a \ge 1$, set
\[
 U_a:=[0,a]\times[0,1]\times[0,a]^{N-2},
 \qquad
 V_a:=[0,1]\times[0,a]\times[0,a]^{N-2},
\]
where the last factor is omitted if $N=2$.  
Since the boxes $U_a$ and $V_a$ are congruent, one has $\lambda_1(U_a) = \lambda_1(V_a)$.
Moreover,
\[
 W_a:=\frac{U_a+V_a}{2}
 =
 [0,(a+1)/2]^2\times[0,a]^{N-2}.
\]
Thanks to Fact~\ref{fa.EiOnBox}, we compute explicitly
\[
 \lambda_1(U_a)=\lambda_1(V_a)
 =
 \pi^2\left(1+\frac{N-1}{a^2}\right) \quad \text{ and } \quad \lambda_1(W_a)
 =
 \pi^2\left(\frac{8}{(a+1)^2}+\frac{N-2}{a^2}\right).
\]
Therefore
\[
 \eta(a):=\frac{\lambda_1(W_a)^{-1/2}}{\lambda_1(U_a)^{-1/2}}  = 
 \left(
 \frac{1+(N-1)a^{-2}}
 {8(a+1)^{-2}+(N-2)a^{-2}}
 \right)^{1/2}.
\]
Notice that the function $\eta$ is continuous on $[1,\infty)$.
Further, it satisfies $\eta(1) = 1$ and
\[
 \eta(a)\to+\infty
 \quad\text{ as }a \to +\infty.
\]
Hence, by the intermediate value theorem, for every $\rho\ge1$ there exists
$a\ge1$ such that $\eta(a)=\rho$.

\medskip

Finally, fix $r>0$ and $\rho\ge1$. 
Choose $a$ with $\eta(a)=\rho$ and set
\[
 \gamma:=\frac{r}{\lambda_1(U_a)^{-1/2}},
 \qquad
 K:=\gamma U_a,
 \qquad
 L:=\gamma V_a.
\]
Then
\[
 \lambda_1(K)^{-1/2} = \lambda_1(L)^{-1/2} = r
 \quad \text{ and } \quad
 \lambda_1\!\left(\frac{K+L}{2}\right)^{-1/2} = \rho r.
\]
Applying the concavity inequality of $f \circ \lambda_1$ with $t=1/2$ gives
\[
 g(\rho r)
 =
 f\left( \lambda_1\!\left(\frac{K+L}{2}\right) \right)
 \ge
 \dfrac{1}{2} \big( f(\lambda_1(K))+ f(\lambda_1(L)) \big) =  g(r).
\]
Since $r>0$ and $\rho\ge1$ are arbitrary, $g$ is nondecreasing on $(0,\infty)$.
Proposition~\ref{prop.BM1} is proven.
\end{proof}

\vspace{0.5cm}

\textbf{Acknowledgement.} 
T. M. L\^e would like to thank Alberto Dom\'inguez Corella for fruitful discussions on the case of $\lambda_1$.
The research of  T. M. L\^e was funded by the Austrian Science Fund (FWF) (10.55776/STA223). K-H. Nguyen-Dang thanks the Morningside Center of Mathematics, Chinese Academy of Sciences, for their support and great working conditions.



\vspace{0.5cm}

\noindent \textsc{Tr\'i Minh L\^E}
\medskip 

\noindent  Faculty of Mathematics,
University of Vienna
\newline Oskar-Morgenstern-Platz 1, 1090 Wien
\medskip
\newline\noindent E-mail: \texttt{tri.minh.le@univie.ac.at}
\newline\noindent\texttt{https://sites.google.com/view/tri-minh-le}
\smallskip\newline\noindent
\noindent Research supported by the Austrian \texttt{FWF} grant \texttt{DOI 10.55776/STA223}.

\bigskip

\noindent \textsc{Khai--Hoan NGUYEN--DANG}

\medskip

\noindent Morningside Center of Mathematics, Chinese Academy of Sciences, Beijing, China
\newline No. 55, Zhongguancun East Road, Haidian District, Beijing 100190 \medskip
\newline\noindent E-mail: \texttt{khaihoann@gmail.com}
\newline\noindent\texttt{https://sites.google.com/view/nguyen-dang-khai-hoan/home} \smallskip\newline
\noindent
\end{document}